\documentclass[11pt]{amsart}
\usepackage{geometry}                
\usepackage{graphicx}
\usepackage{amssymb}
\usepackage{epstopdf}
\newtheorem{theorem}{Theorem}[section]
\newtheorem{lemma}{Lemma}[section]
\newtheorem{corollary}{Corollary}[section]
\newtheorem{proposition}{Proposition}[section]
\newtheorem{conjecture}{Conjecture}[section]
\newtheorem{definition}{Definition}[section]

\numberwithin{equation}{section}

\def\R{\mathbb R}

\def\Om{\Omega}
\def\L{\Lambda}
\def\om{\omega}
\def\d{\partial}

\def\e{\epsilon}
\def\a{\alpha}
\def\b{\beta}
\def\g{\gamma}
\def\l{\lambda}

\title{Contact de Rham cohomology \& Hodge structures transversal to Reeb foliations}

 \author[G.~Katz]{Gabriel Katz}
\address{MIT, Department of Mathematics, 77 Massachusetts Ave., Cambridge, MA 02139, U.S.A.}
\email{gabkatz@gmail.com}

\begin{document}

\maketitle 

\begin{abstract} Let $\beta$ be a contact form on a compact smooth manifold $X$ and $v_\beta$ its Reeb vector field. The paper applies general results of different authors about Hodge structures that are transversal to a given foliation to the special case of $1$-dimensional foliation generated by the Reeb flow $v_\beta$. 
 
The de Rham differential complex $\Omega_{\mathsf{basic}}^\ast(X, v_\beta)$ of, so called, {\sf basic} relative to $v_\beta$-flow differential forms is in the focus of this investigation. By definition, the basic forms vanish when being contracted with $v_\beta$, and so do their differentials. 

We prove that under the change $\beta \leadsto \beta_1 = \beta +df$, where a function $f:X \to \mathbf  R$ such that $df(v_\beta) > -1$, the differential complexes $\Omega_{\mathsf {basic}}^\ast(X, v_{\beta_1})$ and  $\Omega_{\mathsf{basic}}^\ast(X, v_\beta)$ are canonically isomorphic.

We investigate when the $2$-form $d\beta$ and its powers deliver nontrivial elements in the basic de Rham cohomology  $H^\ast_{\mathsf{basic}\,d\mathcal{R}}(X,  v_\beta)$ of the differential complex $\Omega_{\mathsf{basic}}^\ast(X, v_\beta)$. Answers to these questions  contrast sharply in the cases of a closed $X$ and a $X$ with boundary. 

On the other hand, building on work of Ra\'{z}ny \cite{Raz}, we show that on a closed manifold $X$, equipped with a transversal to the Reeb flow Hodge structure that satisfies the {\it  Basic Hard Lefschetz Property},  the basic de Rham cohomology $H^\ast_{\mathsf{basic}\,d\mathcal{R}}(X,  v_\beta)$ are topological invariants of $X$. 
\end{abstract}

\section{Introduction}

In this paper, following \cite{He}, we apply the general results from \cite{EKA}, \cite{EKH} about the  Hodge structures, transversal to foliations, to the case of $1$-dimensional foliations $\mathcal F(v_\b)$ generated by the Reeb flows $v_\b$ on compact connected smooth manifolds $X$. In the case of a closed $X$, we rely heavily on the works of Z. He \cite{He}, Y. Lin \cite{Lin}, and P. Ra\'{z}ny \cite{Raz}. \smallskip

{\sf Erratum.} The earlier version  \cite{K6} of this paper had a basic fault in the proof of Theorem 2.1. Most likely, its claim is also false. As a result, Corollary 2.1 was also erroneous. \hfill $\diamondsuit$ \smallskip

In the present version, we propose a different mechanism (in Definition \ref{def.shallow_chain}, we call it ``{\sf shallow incline exact chain equivalence}") that leads to a somewhat similar assertion.\smallskip

First, let us introduce the main structures we are going to study. Let $\b$ be a contact form on a smooth $(2n+1)$-dimensional manifold $X$. This means that $\b\wedge (d\b)^n > 0$ everywhere. The contact form $\b$  produces a unique  vector field $v_\b$, characterized by two properties: {\sf(1)} $\b(v_\b) = 1$ and {\sf(2)} $v_\b \, \rfloor \, d\b = d\b(v_\b, \sim) \equiv 0$.  This field $v_\b$ is called the {\sf Reeb field} of $\b$. 

A smooth differential $k$-form $\a$ on $X$ is called {\sf basic} with respect to a given vector field $v \neq 0$ if $v \, \rfloor \, \a = 0$ and $v \, \rfloor \, d\a = 0$, where $d\a$ denotes the differential of $\a$. Note that the contact form $\b$ is not basic relatively to its Reeb vector field $v_\b$.

The basic forms on $X$ form the {\sf basic de Rham differential complex}  $\Omega_{\mathsf b}^\ast(X, v_\beta) =_{\mathsf{def}} \Omega_{\mathsf{basic}}^\ast(X, v_\beta)$ (see (\ref{eq.de Rham_complex})) whose cohomology $H^\ast_{\mathsf{basic}\,d\mathcal{R}}(X, v_\b)$ are natural invariants of $v_\b$ and thus of $\b$. Deformations of $\b$, even by conformal factors, may lead to different differential complexes and their different basic de Rham cohomology. 
However, Y.G. Oh proved recently that $H^0_{\mathsf{basic}\,d\mathcal{R}}(X, v_\b) \approx \R \approx H^1_{\mathsf{basic}\,d\mathcal{R}}(X, v_\b)$ for a {\it generic choice} of a contact form $\b$ \cite{Oh}.

\smallskip

Let $\b$ be a contact form and $\tau$ a $(2k)$-form on $X$ such that $v_\b \, \rfloor \,\tau = 0 \text{\, and \,} v_\b \, \rfloor \,d\tau = (d \b)^k.$ Thus, by definition, $\tau$ is basic. It follows that the directional derivative $\mathcal L_{v_\b} \tau = (d \b)^k$. Such $\tau$ can be viewed as a basic anti-derivative of the basic form $(d \b)^k$ in the direction of $v_\b$. 

Let $\{\phi^t: X \to X\}_t$ be a partially defined (because of possible boundary effects) family of diffeomorphisms generated by the vector field $v_\b$. 

 We say that the $k$-form $\tau$ is {\sf  bounded along} a $v_\b$-tajectory $\g$ {\sf in the direction of a $(2k)$-frame} $Fr_\star \in \L^{2k}(T_{x_\star} X / T_{x_\star} \g)$, if 
the values $\big\{\big|\tau\big(\phi^t_\ast(Fr_\star)\big)\big|\big\}_{t \in \R_+}$ are bounded.

In Corollary \ref{cor.tau_d(beta)_BOUNDED}, we prove that  the existence of $\tau$ that  is bounded along all $v_\b$-trajectories for all frames $Fr_\star$, implies that $v_\b$ admits a Lyapunov function $f: X \to \R$ and that the boundary $\d X \neq \emptyset$. \smallskip

In Proposition  \ref{prop.ISO_de Rham} and Theorem \ref{th.CHAIN_equivalent}, we prove that if 
$\b$ and $\b_1 = \b +df$ are two contact forms on a  
smooth manifold $X$, such that $df(v_\b) > -1$, then basic de Rham differential complexes   $\Omega_{\mathsf b}^\ast(X, v_{\beta_1})$ and  $\Omega_{\mathsf b}^\ast(X, v_\beta)$ are canonically isomorphic. Moreover, by Theorem \ref{th.ker_Delta_under_deformations}, if $X$ is a closed manifold, then the spaces $\ker(\Delta^k_{\mathsf{b}}(\b_1))$ and $\ker(\Delta^k_{\mathsf{b}}(\b))$ of basically harmonic $k$-forms, generated by the contact forms $\b_1$ and $\b$, are canonically isomorphic for all $k$.\smallskip

Applying Ra\'{z}ny's Theorem 2.20, \cite{Raz} and Corollary 4.4 \cite{Raz} to our particular setting, in Theorem \ref{th.Lefschetz_implies_iso_homology} we get  the following assertion. Let $\b_0, \b_1$ be two contact forms on a closed manifold $X$ 
so that both Reeb foliations, $\mathcal F(v_{\b_0})$ and $\mathcal F(v_{\b_1})$, possess the {\sf Basic Hard Lefschetz property} (see Theorem \ref{th.Basic_hard_Lefschetz} for its description). 
Then, for each $k$, their basic de Rham $k$-cohomologies  are isomorphic $$H^k_{\mathsf{basic}\,d\mathcal{R}}(X,  v_{\b_1}) \approx H^k_{\mathsf{basic}\,d\mathcal{R}}(X,  v_{\b_0}),
$$ 
 and so are the spaces of basic symplectically harmonic $k$-forms:
 $$\ker(\Delta_{d\b_1}^k(v_{\b_1})) \approx \ker(\Delta_{d\b_0}^k(v_{\b_0})). \quad \diamondsuit$$

\section{Contact de Rham cohomology of the Reeb foliations}

 When we consider smooth structures/objects  (like functions, differential forms, vector fields, etc.) on a given smooth compact manifold $X$ {\it with boundary}, this means that $X$ is embedded into a larger open equidimensional manifold $\hat X$ so that a structure in question is a restriction of a similar structure/object on $\hat X$ to $X$. As we change the structure/object on $X$, we may shrink the ambient $\hat X$. Thus one should think of the ambient $\hat X \supset X$ as a germ.\smallskip

 Let $X$ be a compact smooth $m$-dimensional manifold.  Consider the de Rham differential complex of differential forms
\begin{eqnarray} \label{eq.de Rham_complex_no_field}
\quad   \Om^\ast_{d\mathcal{R}}(X) =_{\mathsf{def}} \quad \quad \quad \quad \\
\Big\{ 0 \stackrel{d}{\rightarrow} 
\Om^0(X) \stackrel{d}{\rightarrow} \Om^1(X) \stackrel{d}{\rightarrow}  \Om^2(X) \stackrel{d}{\rightarrow} \ldots \stackrel{d}{\rightarrow}  \Om^{m}(X) \stackrel{d}{\rightarrow} 0\Big\}, \nonumber
\end{eqnarray}
where  $\Om^0(X) \approx C^\infty(X, \R)$ is the algebra of smooth functions on $X$. 
\smallskip

For a smooth non-vanishing vector field $v$ on a 
$m$-dimensional manifold $X$, consider the vector space $\Om^k_{\mathsf b}(X, v)$ of smooth $v$-{\sf invariant} exterior differential $k$-forms $\a$ that are $v${\sf-horizontal}, that is,  $v \rfloor \a = 0$. The differential $d\a$ is evidently $v$-invariant when $\a$ is.  Hence, for a $v$-invariant horizontal $\a$, we get $$0 = \mathcal L_v \a = v\, \rfloor\, d\a + d(v\, \rfloor\, \a) = v \,\rfloor \,d\a,
$$ which implies that $d\a$ is horizontal as well. When $k=0$, then $\Om^0_{\mathsf b}(X, v) \approx C^\infty(X, v)$, the space of $v$-invariant smooth functions on $X$. 

Alternatively, we could define $\Om^k_{\mathsf b}(X, v)$ as the vector space of $k$-forms $\a$ that are $v$-horizontal and whose differential $d\a$ is $v$-horizontal. Some authors call such differential forms ``{\sf basic}".

Therefore, for $m = \dim(X)$ and any vector field $v \neq 0$ on $X$, we get the {\sf horizontal and $v$-invariant} (alternatively, {\sf basic}) de Rham differential complex 
\begin{eqnarray} \label{eq.de Rham_complex}
\quad   \Om^\ast_{\mathsf{basic}\,d\mathcal{R}}(X, v) =_{\mathsf{def}} \quad \quad \quad \quad \\
\Big\{ 0 \stackrel{d}{\rightarrow} 
\Om^0_{\mathsf b}(X, v) \stackrel{d}{\rightarrow} \Om^1_{\mathsf b}(X, v) \stackrel{d}{\rightarrow}  \Om^2_{\mathsf b}(X, v) \stackrel{d}{\rightarrow} \ldots \stackrel{d}{\rightarrow}  \Om^{m-1}_{\mathsf b}(X, v) \stackrel{d}{\rightarrow} 0\Big\} \nonumber
\end{eqnarray}
of $\Om^0_{\mathsf b}(X, v)$-modules. Here the algebra $\Om^0_{\mathsf b}(X, v)$ consists of smooth functions $f: X \to \R$ such that the directional derivative $\mathcal L_v f = 0$.
Note that $\Om^{m}_{\mathsf b}(X, v) = 0$ since no  non-zero $m$-form on an open set can be $v$-horizontal.\smallskip

Evidently, the de Rham differential complex of basic forms $\Om^\ast_{\mathsf{basic}\,d\mathcal{R}}(X, v)$ maps canonically to the de Rham differential complex $\Om^\ast_{d\mathcal{R}}(X)$ of all forms. 

\begin{definition}\label{def.de Rham}
We denote by $H^\ast_{\mathsf{basic}\,d\mathcal{R}}(X, v)$ the cohomology groups of the $v$-basic de Rham differential complex (\ref{eq.de Rham_complex}).
\hfill $\diamondsuit$
\end{definition}

\begin{lemma} \label{lem.diffeo_of_de Rham} 
Let $\Psi: X \to X$ be a smooth diffeomorphism and  $v_1, v_2$ 
smooth non-vanishing vector fields on $X$ such that $(\Psi)_\ast(v_1) = \l \cdot v_0$ for a smooth  positive function $\l: X \to \R$. 

Then the  basic de Rham differential complexes $\Om^\ast_{\mathsf{basic}\,d\mathcal{R}}(X, v_0)$ and $\Om^\ast_{\mathsf{basic}\,d\mathcal{R}}(X, v_1)$ are isomorphic with the help of $\Psi^\ast$. Hence, the de Rham basic cohomology are isomorphic: $$\Psi^\ast: H^\ast_{\mathsf{basic}\,d\mathcal{R}}(X,  v_0) \approx H^\ast_{\mathsf{basic}\,d\mathcal{R}}(X,  v_1).$$
\end{lemma}



\begin{proof} For a given $X$, the notion $\{v \rfloor \a =0,\, v \rfloor d\a =0\}$ of a basic form $\a$ depends evidently only on the conformal class of the vector field $v$.  Therefore, using that  $d\Psi^\ast(\a) = \Psi^\ast(d\a)$ and that $(\Psi)_\ast(v_1) = \l v_0$, we conclude that $\Psi^\ast$ is an isomorphism that maps basic forms to basic forms and commutes with their differentials. Thus, the basic differential complex  $\Om^\ast_{\mathsf{basic}\,d\mathcal{R}}(X, v_0)$ is isomorphic to the basic differential complex  $\Om^\ast_{\mathsf{basic}\,d\mathcal{R}}(X, v_1)$ with the help of $\Psi^\ast$.
\hfill
\end{proof}


\begin{lemma} \label{lem.contact_de Rham_A} 
Let $\b$ be a contact form on a 
smooth manifold $X$ and $v_\b$ its Reeb vector field. Consider a $1$-form $\b_1 = \b + d f$, where the smooth function $f: X \to \R$ is such that  $df(v_\b) > -1$ (i.e., $f$ is not ``negatively very  steep" along the Reeb trajectories). \smallskip

$\bullet$ With $\b$ being fixed, the inequality $df(v_\b) > -1$ defines a nonempty open convex set $\mathcal C(\b)$ in the functional space $C^\infty(X, \R)$.

$\bullet$ The form $\b_1$ is a contact form on $X$. 
%
\end{lemma}

\begin{proof} Since $\b_1 = \b + d f$, $$\b_1\wedge (d\b_1)^n = \b_1 \wedge (d\b)^n = \b \wedge (d\b)^n + df \wedge (d\b)^n.$$ 

We claim that the $(2n+1)$-form $\b_1\wedge (d\b_1)^n$ is positive on $X$, provided $df(v_\b) > -1$. Indeed, the $(2n+1)$-form $df \wedge (d\b)^n$ is proportional to $(2n+1)$-form $\b \wedge (d\b)^n$ with some functional coefficient $h: X \to \R$. Thus $\b_1\wedge (d\b_1)^n > 0$ is equivalent to the inequality $1+ h > 0$ on $X$. 

To detect $h$, it suffices to compare the contractions of the two forms,  $df \wedge (d\b)^n = h (\b \wedge (d\b)^n)$ and $\b \wedge (d\b)^n$, with a non-vanishing Reeb vector field $v_\b$. Recall that $v_\b \rfloor\b =1$ and $v_\b \rfloor d\b \equiv 0$.
Using these properties of $v_\b$, we get $$v_\b \rfloor (\b \wedge (d\b)^n) = (d\b)^n, \text{ while } v_\b \rfloor (df \wedge (d\b)^n) = df(v_\b) (d\b)^n = h\;(d\b)^n.$$ Thus, when $df(v_\b) > -1$, $v_\b \rfloor(\b_1\wedge (d\b_1)^n)$ is positively proportional to $(d\b)^n$ with the functional coefficient $h$. Therefore, $\b_1$ is a contact form, provided $df(v_\b) > -1$.

\smallskip
Put $v_{\b_1} =_{\mathsf{def}} (1 + df(v_\b))^{-1} v_\b$. Let us check that $v_{\b_1}$ is the Reeb vector field for the form $\b_1$. Indeed, $\b_1(v_{\b_1}) = \b_1( (1 + df(v_\b))^{-1} v_\b) = (\b +df) \big((1 + df(v_\b))^{-1} v_\b\big) =1$ since $\b(v_\b) = 1$, and $d\b_1(v_{\b_1}) = d\b( (1 + df(v_\b))^{-1} v_\b) = (1 + df(v_\b))^{-1} d\b(v_\b) \equiv 0.$ \smallskip

If $df(v_\b) > -1$ and $dg(v_\b) > -1$ for some $f, g \in C^\infty(X, \R)$, then $d(t f + (1-t) g) >  -1$ for any $t \in [0,1]$. Therefore, $\mathcal C(\b) \subset C^\infty(X, \R)$ is an open convex set that contains all constant functions $c$ (since $dc(v_\b) = 0 > -1$).
\hfill
\end{proof}

Although the deformation $\b \leadsto \b_1=\b+df$, where $df(v_\b) > -1$, preserves the conformal class of the Reeb field $v_\b$, it may change the original contact distribution $\xi_\b = \ker(\b)$. The function $df(v_\b)$ on $X$ and its norms may be viewed as  quantities that distinguish between the distributions $\xi_\b$ and $\xi_{\b_1}$.\smallskip

Let $\mathsf{Cont}(X) \subset \Omega^1(X)$ denote the space of contact forms $\b$. Since $\b \wedge (d\b)^n > 0$ on $X$ is an open constraint, $\mathsf{Cont}(X)$ is open in $\Omega^1(X)$. Its boundary consists of $1$-forms $\b$ such that $\b \wedge (d\b)^n \geq 0$ on $X$ and $(\b \wedge (d\b)^n)(x) =0$ for some $x \in X$.

\begin{proposition} \label{prop.ISO_de Rham}
Let $\b$ be a contact form on a smooth manifold $X$ and $v_\b$ its Reeb field. Consider a $1$-form $\b_1 = \b + d f$, where the smooth function $f: X \to \R$ is such that  $df(v_\b) > -1$. 

Then the de Rham basic differential complexes $\Om^\ast_{\mathsf{basic}\,d\mathcal{R}}(X, v_{\b_1})$ and $\Om^\ast_{\mathsf{basic}\,d\mathcal{R}}(X, v_{\b})$ are canonically isomorphic. 
As a result, the de Rham basic cohomology spaces 
are isomorphic as well: $$H^\ast_{\mathsf{basic}\,d\mathcal{R}}(X,  v_{\b_1}) \approx H^\ast_{\mathsf{basic}\,d\mathcal{R}}(X,  v_{\b}).$$

In other words, the canonical isomorphisms $\Om^\ast_{\mathsf{basic}\,d\mathcal{R}}(X, v_{\b_1}) \approx \Om^\ast_{\mathsf{basic}\,d\mathcal{R}}(X, v_{\b})$ are valid for any  form $\b_1$ residing in a non-empty 
convex set $$\mathcal O(\b) = \{ \b + df\big |\; f \in C^\infty(X, \R),\; df(v_\b) > -1\} \subset \mathsf{Cont}(X).$$  
\end{proposition}

\begin{proof} 


By their definitions, the de Rham basic differential complexes $\Om^\ast_{\mathsf{basic}\,d\mathcal{R}}(X, v_{\b_1})$ and $\Om^\ast_{\mathsf{basic}\,d\mathcal{R}}(X, v_{\b})$ depend only on the conformal classes of the Reeb vector fields $v_{\b_1}$ and $v_{\b}$, respectively. By Lemma \ref{lem.contact_de Rham_A}, the two fields are positively proportional.  Therefore, we conclude that the de Rham basic differential complexes
 $\Om^\ast_{\mathsf{basic}\,d\mathcal{R}}(X, v_{\b_1})$ and $\Om^\ast_{\mathsf{basic}\,d\mathcal{R}}(X, v_{\b})$ are canonically isomorphic, as long as $f \in \mathcal C(\b) \subset C^\infty(X, \R)$. 
\hfill
\end{proof}

\begin{definition}\label{def.shallow_chain}  Two contact forms $\b$ and $\b'$ on $X$ are called {\sf  shallow-incline-exact-chain-equivalent} if there exist a sequence of contact forms $\{\b_i\}_{i \in [0, N]}$ such that $\b_0 =\b$, $\b' = \b_N$, and each $\b_{i+1} = \b_i + df_i$, where $f_i \in \mathcal C(\b_i)$, for all $i \in [0, N-1]$. \smallskip

We denote by $\mathcal E(\b)$ the space of all contact forms that are shallow-incline-exact-chain-equivalent to $\b$.
\hfill $\diamondsuit$
\end{definition}

The next lemma shows that this definition of equivalence does not lead to new extensions $\mathcal E(\b)$ of the already familiar set $\mathcal O(\b)=_{\mathsf{def}}\b + d(\mathcal C(\b))$.

\begin{lemma} \label{lem.VACUUS_ext} 
For a given contact form $\b$, the contact forms $\b' \in \mathcal E(\b)$ 
that are shallow-incline-exact-chain-equivalent to $\b$ belong to the open and convex set $\mathcal O(\b)=_{\mathsf{def}}\b + d(\mathcal C(\b))$ in the affine space $\mathcal D(\b)$ of all $1$-forms on $X$ that differ from $\b$ by an exact $1$-form. 
\end{lemma}

\begin{proof} 
Consider a shallow-incline-exact chain $\{\b_{i+1} = \b_i + df_i\}_{i \in [0, N-1]}$, where $f_i \in \mathcal C(\b_i)$ and $\b_0 =\b$. 
Then a direct computation shows that 
\begin{eqnarray}\label{eq.SUM}
v_{\b'} = \big(1+ \sum_i df_i(v_\b)\big)^{-1} v_\b,
\end{eqnarray}
which implies that $1+ \sum_i df_i(v_\b) > 0$ since all intermediate proportionality coefficients $(1+ df_i(v_{\b_i}))^{-1}$ next to $v_{\b_i}$ are positive. As an example, let us execute the first steps in this computation. Put $a_i = df_i(v_\b)$. Then, by Lemma  \ref{lem.contact_de Rham_A}, $v_{\b_1} = \frac{1}{1+a_1} v_\b$. Similarly, 
$$v_{\b_2} =  \frac{1}{1+df_2(v_{\b_1})}\; v_{\b_1} = \frac{1}{1+df_2(\frac{1}{1+a_1} v_\b)} \cdot \Big(\frac{1}{1+a_1}\; v_\b\Big) = \Big( \frac{1}{1+ \frac{1}{1+a_1}\,a_2} \Big)\Big(\frac{1}{1+a_1}\; v_\b\Big)$$ 
$$= \frac{1}{1 +a_1+a_2}\, v_\b.$$

Letting $f_\bullet = \sum_i  f_i$ in (\ref{eq.SUM}), we conclude that $f_\bullet(v_\b) > -1$. So, $\b_N \in \mathcal O(\b) = \b + d(\mathcal C(\b))$. Thus, $\mathcal E(\b) \subset \mathcal O(\b)$. On the other hand, $\mathcal E(\b) \supset \mathcal O(\b)$ by definition. As a result, $\mathcal E(\b) = \mathcal O(\b)$, an open and convex set in $\mathcal D(\b)$.
\hfill
\end{proof}

\begin{lemma}\label{lem.O=O'} If, for some $\b, \b' \in \mathsf{Cont}(X)$,  $\mathcal O(\b) \cap  \mathcal O(\b') \neq \emptyset$, then $\mathcal O(\b) =  \mathcal O(\b')$.
\end{lemma}

\begin{proof} Pick $\b_1 \in \mathcal O(\b) \cap  \mathcal O(\b')$. Then $\b_1 = \b + df$, where $df(v_\b) > -1$. Also, $\b_1 = \b' + df'$, where $df'(v_{\b'}) > -1$. It follows that $\b' =\b_1- df'$,  where $-df'(v_{\b_1}) > -1$. Since $\b$ and $\b'$ are shallow-incline-exact-chain-equivalent, by Lemma \ref{lem.VACUUS_ext}, $\b' \in \mathcal O(\b)$.  Similarly, $\b \in \mathcal O(\b')$. Therefore, $\mathcal O(\b) =  \mathcal O(\b')$.
\hfill
\end{proof}

Lemma \ref{lem.O=O'}, in combination with Lemma \ref{lem.contact_de Rham_A}, has the following direct implication: the space $\mathsf{Cont}(X)$ ``fibers" over another space so that the fibers are convex open sets $\mathcal O(\b)$ in the vector space  $d(\Om^0(X))$.

\begin{corollary} Assume that $\mathsf{Cont}(X) \neq \emptyset$. Consider the continuous quotient map $Q: \Om^1(X) \to \Om^1(X)/d(\Om^0(X))$ of topological vector spaces and its restriction $Q^\dagger$ to the open subset $\mathsf{Cont}(X) \subset \Om^1(X)$. 
For each element  $\b \in \mathsf{Cont}(X)$, let $\a = Q(\b)$. \smallskip

Then the fiber $(Q^\dagger)^{-1}(\a)$ coincides with $\mathcal O(\b)$, an open and convex set in the affine space $(Q)^{-1}(\a)$. 
\hfill $\diamondsuit$
\end{corollary}

\begin{theorem}\label{th.CHAIN_equivalent} If two contact forms, $\b$ and $\b'$, are shallow-incline-exact-chain-equivalent (see Definition \ref{def.shallow_chain}), 
then  the de Rham basic differential complexes $\Om^\ast_{\mathsf{basic}\,d\mathcal{R}}(X, v_{\b'})$ and $\Om^\ast_{\mathsf{basic}\,d\mathcal{R}}(X, v_{\b})$ are canonically isomorphic.
\end{theorem}

\begin{proof} If $\b$ and $\b'$ are shallow incline exact chain equivalent, by Proposition \ref{prop.ISO_de Rham}, we get canonical isomorphisms $\Om^\ast_{\mathsf{basic}\,d\mathcal{R}}(X, v_{\b_{i+1}}) \approx \Om^\ast_{\mathsf{basic}\,d\mathcal{R}}(X, v_{\b_i})$
for all $i \in [0, N-1]$. By an induction in $i$ and Lemma \ref{lem.VACUUS_ext}, the claim follows. \smallskip
\hfill
\end{proof}


\smallskip

By Lemma \ref{lem.diffeo_of_de Rham}, the orientation preserving group $\mathsf{Diff}_+(X)$ of smooth diffeomorphisms on $X$  acts on the subspace $\mathsf{Cont}(X) \subset \Om^1(X)$. The $\mathsf{Diff}_+(X)$-action preserves also the subspace $d(\Om^0(X)) \subset \Om^1(X)$. Moreover, the $\mathsf{Diff}_+(X)$-action on $\Om^1(X)$ maps each fiber of the map $Q: \Om^1(X) \to \Om^1(X)/d(\Om^0(X))$ to a $Q$-fiber. 
 Therefore, Theorem \ref{th.CHAIN_equivalent} and Lemma \ref{lem.diffeo_of_de Rham} imply the following claim.
\begin{corollary}\label{cor.from_contact_to} 
The group $\mathsf{Diff}_+(X)$ of the orientation preserving smooth diffeomorphisms acts on the space $\mathsf{Cont}(X)$ and on its quotient $Q^\dagger(\mathsf{Cont}(X))$ so that the map $Q^\dagger$ is equivariant. 
 
Let $\mathbf \Om_{\mathsf{basic}\, d\mathcal R}(X; \mathsf{Reeb})$ be the set of isomorphism classes of basic de Rham complexes, produced by various Reeb vector fields on $X$. 

There exists a surjective map 
$$\Upsilon: \mathsf{Cont}(X)/\mathsf{Diff}_+(X) \to Q^\dagger(\mathsf{Cont}(X))/\mathsf{Diff}_+(X) \to \mathbf \Om_{\mathsf{basic}\, d\mathcal R}(X;  \mathsf{Reeb}).$$
\hfill $\diamondsuit$ 
\end{corollary}

\section{Contact de Rham cohomology, Lyapunov functions/forms of the Reeb field $v_\b$, and  the class $[d\b] \in H^2_{\mathsf{basic}\,d\mathcal{R}}(X, v_\b)$}

In this section, we compare the contact de Rham cohomology, in particular, the behavior of the classes $([d\b])^\ell \in H^{2\ell}_{\mathsf{basic}\,d\mathcal{R}}(X, v_\b)$ on {\it closed} and {\it compact manifolds with boundary}. The boundary effects generate a sharp dichotomy between these two settings.\smallskip

Recall that  when considering smooth structures/objects on a given smooth compact manifold $X$ {\it with boundary}, we embed  $X$ into a larger open equidimensional manifold $\hat X$ so that a structure in question is a restriction of a similar structure/object on $\hat X$. We think of the ambient $\hat X \supset X$ as a germ.

\begin{lemma}\label{lem.NO_Lyapunov_for_closed} Let $v$ be a non-vanishing vector field on a smooth manifold $X$, possibly with boundary. Let $\g$ be a positive/negative time $v$-trajectory originating at a point in the interior of $X$ such that its closure $\bar\g \subset X$ is a compact set, contained in the interior of $X$. 

We call such trajectory ``{\sf trapped}".\smallskip

Then $v$ does not admit any Lyapunov function $f: X \to \R$ such that $df(v) > 0$ in $X$.
\end{lemma}

\begin{proof}  Assume, to the contrary, that $v$ admits a Lyapunov function $f: X \to \R$. Since $\bar\g$ is compact, $f$ must attend its maximum at some point $x_\star \in \bar\g$. Since $\bar\g$ is a positive time $v$-invariant set, it contains the positive time $v$-trajectory $\g_\star$ through $x_\star$. However, since $df(v)(x_\star) > 0$, the point $x_\star$ cannot deliver the maximum of $f$ on $\bar\g$.
\hfill
\end{proof}

\begin{proposition}\label{prop.[dBETA]not_zero} Let $\b$ be a contact form on a  smooth compact connected manifold $X$, and let $v_\b$ be its Reeb vector field. \smallskip

$\bullet$ The homology class $[d\b] \in H^2_{\mathsf{basic}\,d\mathcal{R}}(X, v_\b)$ is zero if and only if $v_\b$ admits a closed $1$-form $\a$ such that $\a(v_\b) = 1$. 

If $[d\b] = 0$ in $H^2_{\mathsf{basic}\,d\mathcal{R}}(X, v_\b)$, then the $v_\b$-flow is transversal to a codimension one foliation $\mathcal F_\a$ (well-defined in the vicinity of $X$ in $\hat X$).
\smallskip

$\bullet$ If $v_\b$ admits a Lyapunov function, then $\d X \neq \emptyset$ and the homology class $[d\b] \in H^2_{\mathsf{basic}\,d\mathcal{R}}(X, v_\b)$ is zero. \smallskip

Conversely, assuming that $H^1(X; \R) =0$, if  $[d\b] = 0$ in $H^2_{\mathsf{basic}\,d\mathcal{R}}(X, v_\b)$, then $v_\b$ admits a Lyapunov function, and the boundary $\d X \neq \emptyset$. 
\end{proposition}

\begin{proof} If $X$ is compact and a closed Lyapunov form $\a$ is such that $\a(v_\b) = 1$, then the $2$-form 
 $d\b = d(\b - \a)$ is a boundary of a $v_\b$-\emph{horizontal} $1$-form $\b - \a$. Indeed, $(\b - \a)(v_\b) = 1-1 =0$. The boundary $d\b$ is also horizontal by the basic properties of Reeb vector fields. So, $\b - \a$ is a basic form. As a result, $[d\b] = 0$ in $H^2_{\mathsf{basic}\,d\mathcal{R}}(X, v)$. \smallskip

Conversely,  if $[d\b] = 0$ in $H^2_{\mathsf{basic}\,d\mathcal{R}}(X, v)$, then for a basic $1$-form $\a'$, we get $d\b = d\a'$. Thus, $(\b -\a')(v_\b) = \b(v_\b) =1$. Therefore, the $1$-form $\a = \b-\a'$ is closed and  $\a(v_\b) =1$. 

Let $\mathcal F_\a$ be codimension one foliation, defined in the vicinity of $X$ in $\hat X$ by the integrable distribution $\{u \in TX|\; \a(u) = 0\}$, where $\a$ is as above. 
Hence, if $[d\b] = 0$ in $H^2_{\mathsf{basic}\,d\mathcal{R}}(X, v_\b)$, then the $v_\b$-flow  is transversal to $\mathcal F_\a$, since $\a(v_\b) =1$. \smallskip

This proves the claims in the first bullet.
\smallskip

If $X$ is compact and $df(v_\b) > 0$ for a smooth Lyapunov function $f$, then by \cite{K5}, for another Lyapunov function $h$, we have $dh(v_\b) = 1$.  Therefore, $d\b = d(\b - dh)$ is a boundary of a $v_\b$-\emph{horizontal} $1$-form $\a = \b - dh$. Its differential $d\b$ is also horizontal by the basic properties of Reeb vector fields; so, $\b - dh$ is a basic form. As a result, $[d\b] = 0$ in $H^2_{\mathsf{basic}\,d\mathcal{R}}(X, v)$.\smallskip

Assume now that $H^1(X; \R) =0$. 
Since $d(d\b)= 0$ and $v_\b\, \rfloor \, d\b =0$, the form $d\b$ is a cocycle in $\Om^2_{\mathsf{basic}\,d\mathcal{R}}(X, v)$. If $[d\b]=0$ in $H^2_{\mathsf{basic}\,d\mathcal{R}}(X, v)$, then there is a \emph{basic} $1$-form $\a$ such that $d\a = d\b$. Using that $H^1(X; \R) =0$, we get $\b -\a = df$ for a smooth function $f: X \to \R$. Since $\a$ is a $v_\b$-basic form, then $df(v_\b) = \b(v_\b) - \a(v_\b) = 1-0 = 1$. Thus, $df$ is a Lyapunov function for $v_\b$. 

If $X$ is a closed manifold and $v_\b \neq 0$ admits a Lyapunov function $f$, then there exists a point $x_\star \in X$ where $f$ attends its extremum. Thus, $df(v_\b(x_\star)) = 0$, which contradicts to $f$ being a Lyapunov function. 
Therefore, $\d_1 X \neq \emptyset$. \hfill
\end{proof}

\begin{corollary}\label{cor.d(beta)is_homology_nontrivial} Let $\b$ be a smooth contact form on a \emph{closed} manifold $X$. Assume that $H^1(X; \R) = 0$.
Then the class $[d\b] \in H^2_{\mathsf{basic}\,d\mathcal{R}}(X, v)$ is nontrivial, while the class $[d\b] \in H^2_{d\mathcal{R}}(X; \R)$ is trivial.
\end{corollary}

\begin{proof} 
Since $X$ is closed, the closure of any $v_\b$-trajectory is compact in $X = \mathsf{int}(X)$. By Lemma \ref{lem.NO_Lyapunov_for_closed}, $v_\b$ does not admit a Lyapunov function.

By Proposition \ref{prop.[dBETA]not_zero} (see the last paragraph in the lemma proof), 
the class $[d\b] \in H^2_{\mathsf{basic}\,d\mathcal{R}}(X, v)$ is nontrivial, provided  $H^1(X; \R) = 0$. At the same time, the class $[d\b] \in H^2_{d\mathcal{R}}(X; \R)$ is always trivial.
\hfill
\end{proof}

Recall that the vector  field $v_\b$, tangent to the fibers of the Hoph fibration $S^3 \to S^2$, is not transversal to {\it any} $2$-dimensional foliation on $S^3$.  Thus,  $[d\b] \neq 0$ in $H^2_{\mathsf{basic}\,d\mathcal{R}}(S^3, v_\b)$. 

\begin{definition}\label{def.LINKING_prop} Let $X$ be an oriented closed $3$-fold, and $\phi$ a nonsingular vector flow on $X$. Following \cite{Good}, we say that $\phi$ has the {\sf linking property}, if a periodic $\phi$-orbit $\g$  bounds an embedded disk $D^2 \subset X$, then the interior of $D^2$ intersects another periodic $\phi$-orbit.
\hfill $\diamondsuit$
\end{definition}

Note that if a flow $\phi$ has a closed trajectory, bounded by an embedded $2$-disk, and $\phi$ possesses the linking property, then $\phi$ must have at least two closed trajectories.

\begin{corollary} Let $X$ be an oriented closed $3$-fold equipped with a contact form $\b$. 

If $[d\b] = 0$ in $H^2_{\mathsf{basic}\,d\mathcal{R}}(X, v_\b)$, then the $v_\b$-flow has the linking property. 
\end{corollary}


\begin{proof}
By \cite{Good}, Theorem 1.2, if a nonsingular  vector flow $\phi$ on an oriented closed $3$-fold is transversal to a $2$-dimensional foliation, then $\phi$ has the linking property. Therefore, by Proposition \ref{prop.[dBETA]not_zero}, if $[d\b] = 0$ in $H^2_{\mathsf{basic}\,d\mathcal{R}}(X, v_\b)$, then the $v_\b$-flow has the linking property, since $v_\b$ is transversal to the foliation $\mathcal F_\a$, where the closed $1$-form $\a = \b - \a'$ is as in the proof of Proposition \ref{prop.[dBETA]not_zero}. 

In particular, if $[d\b] = 0$ and the Reeb flow has a closed trajectory, bounded by an embedded disk, then it has another closed trajectory, linked with the first one.
\hfill
\end{proof}

By \cite{Good}, Theorem 2.1, in any homotopy class of nonsingular vector fields on any oriented closed $3$-fold, there is an open set of nonsingular vector fields {\it without} the linking property. In fact, a similar claim is valid in the higher dimensions. These results motivate the following conjecture.

\begin{conjecture} Let $X$ be a closed (perhaps, compact?) connected oriented smooth manifold. The set of contact forms $\b$ on $X$ for which $[d\b] \neq 0$ in $H^2_{\mathsf{basic}\,d\mathcal{R}}(X, v_\b)$ is open in the space of all contact forms. \hfill $\diamondsuit$
\end{conjecture}




Given a non-vanishing vector field $v$, a closed $1$-form $\a$ is called {\sf Lyapunov}, if  $\a(v) > 0$.\smallskip

Recall the notion of {\sf Schwartzman's asymptotic cycles $A(v_\b, \mu)$} (see \cite{F}, page 199), as it applies to the case of Reeb flow $v_\b$ on closed manifold $X$.
Consider the $v_\b$-invariant Borel measure $\mu_\b$ on $X$, defined by the volume form $\b \wedge (d\b)^n$. We define an element  $A(v_\b, \mu_\b) \in H_1(X; \R)$ by the formula 
$$\langle A(v_\b, \mu_\b), [\a] \rangle =_{\mathsf{def}} \int_X \a(v_\b)\; \b \wedge (d\b)^n,
$$
where $[\a] \in H^1(X; \R)$ runs over the de Rham cohomology classes and  $\a$ is a closed $1$-form of the class $[\a]$. 
In fact, this definition of $A(v_\b, \mu_\b)$ does not depend on the choice of $\a$ in its cohomology class (see \cite{Sch}); it depends only on  the contact form  $\b$.  \smallskip

As a special case of Proposition 10.44 from \cite{F}, we get the following claim.

\begin{proposition} Let $X$ be a closed smooth manifold equipped with a contact form $\b$. 

If a closed $1$-form $\a$ is Lyapunov for the Reeb flow $v_\b$, then it satisfies the homological  property:
$\langle A(v_\b, \mu_\b),\, [\a] \rangle < 0,$
where $A(v_\b, \mu_\b) \in H_1(X; \R)$.

Moreover, a similar inequality is valid for any $v_\b$-invariant positive Borel measure $\mu$ on the manifold $X$.
\hfill $\diamondsuit$
\end{proposition}

\smallskip


Let $X$ be a compact connected smooth manifold, equipped with a non-vanishing vector field $v$. Given a non-vanishing (basic) closed $k$-form $\eta$ such that $v \, \rfloor \, \eta = 0$, we ask whether there exist a (basic) $k$-form $\tau$ such that $v \, \rfloor \, \tau = 0$ and $\mathcal L_{v}(\tau)  = \eta$, and what are the implications of its existence? Informally speaking, we are asking for implications of the existence of the ``global integral" or `` global potential"  $\mathcal L_{v}^{-1}(\eta)$  for a given basic form $\eta$.

\begin{definition}\label{def.tau_correlated_with_eta} Let $X$ be a smooth manifold equipped with a non-vanishing vector field $v$. 

 Let $Fr_\star$ be an ordered collection of $k$ tangent to $X$ linearly independent vectors at a point $x_\star  \in \mathsf{int}(X)$, the vectors being transversal to the positive time trajectory $\g$ through $x_\star$.  
 
 We consider  the $k$-frame field $Fr^\bullet_{\g}$, spread along $\g$ by the differential $\phi^t_\ast$ of the positive time $v$-flow $\{\phi^t\}_{t \in \R_+}$,  applied to $Fr_\star$. We may view $Fr^\bullet_\g$ as a section of the pull-back of the bundle $\{\L^k(TX/T\g)|_\g \to \g\}$ under the parametrization map $\phi^t: \R \to \g$, where $\phi^t(0) = x_\star$.   

Let $\tau$ be a horizontal  (i.e., $v \, \rfloor \, \tau = 0$) differential $k$-form on $X$.\smallskip

$\bullet$ We say that the $k$-form $\tau$ is {\sf  bounded along} $\g$ {\sf in the direction of the frame} $Fr_\star$, if 
the values $\big\{\big|\tau\big(\phi^t_\ast(Fr_\star)\big)\big|\big\}_{t \in \R_+}$ are bounded. 
\hfill $\diamondsuit$
\end{definition}

\begin{lemma}\label{lem.Lv(tau) = eta} Let $X$ be a compact connected smooth $d$-manifold, equipped with a non-vanishing vector field $v$. Let $\eta$ be a $v$-basic non-vanishing $k$-form and $\tau$ a $v$-horizontal $k$-form on $X$  
such that $\mathcal L_v\tau = \eta$. 

Then $\tau$ is \emph{not}   bounded along any trapped $v$-trajectory $\g \subset \mathsf{int}(X)$ in the direction of any $k$-frame $Fr_\star$ such that $\eta(Fr_\star) \neq 0$. In fact, the function $\{t \to \tau\big(\phi^t_\ast(Fr_\star)\big)|_\g\}$ is a polynomial of degree $1$.

On the other hand, for any $v$-trajectory $\g \subset \mathsf{int}(X)$ and any $k$-frame $Fr_\star$ such that $\eta(Fr_\star)= 0$, the values $\{\tau\big(\phi^t_\ast(Fr_\star)\big)|_\g\}$ are $t$-independent.
%
\end{lemma}

\begin{proof} 
Contrary to the claim, assume that $\tau$ is  bounded along a trapped $v$-trajectory $\g \subset \mathsf{int}(X)$ in the direction of a frame $Fr_\star$ such that $\eta(Fr_\star) > 0$. We denote by $\g_T$ the portion of $\g$ that corresponds to the time interval $[0, T]$. Then, for a $v$-invariant frame $Fr$ along $\g_T$, using that $(\mathcal L_{v}\tau)(Fr)  = \eta(Fr)$, we get 
$$\int_{[0, T]} \mathcal L_{v}\tau\big(\phi^t_\ast(Fr_\star)\big)\, dt = \int_{[0, T]} \eta\big(\phi^t_\ast(Fr_\star)\big)\, dt$$
Using that both the form $\eta$ and the frame $Fr^\bullet_\g$ are $\phi^t$-invariant, 
$$\int_{[0, T]} \mathcal L_{v}\tau\big(\phi^t_\ast(Fr_\star)\big)\, dt = \eta\big(Fr_\star\big) \cdot T.$$
Since $$\mathcal L_{v}\tau\big(\phi^t_\ast(Fr_\star)\big) = \lim_{\e \to 0}\frac{\tau\big(\phi^{t+\e}_\ast(Fr_\star)\big) - \tau\big(\phi^t_\ast(Fr_\star)\big)}{\e} = \frac{d}{dt} \tau\big(\phi^t_\ast(Fr_\star)\big),$$
we get
$$\tau\big(\phi^T_\ast(Fr_\star)\big) -  \tau\big(Fr_\star\big) = \eta\big(Fr_\star\big) \cdot T.$$
Thus, 
$$\tau\big(\phi^T_\ast(Fr_\star))\big) =  \tau\big(Fr_\star\big) +  \eta\big(Fr_\star\big) \cdot T$$
is a polynomial of degree $1$ in $T$, unless $\eta\big(Fr_\star\big) = 0$.
We see that $\lim_{T \to \infty} \tau\big(\phi^T_\ast(Fr_\star)\big) = \infty$, provided $\eta(Fr_\star) \neq 0$; so $\tau$ is not  bounded along any trapped $v$-trajectory $\g$ in the direction of any frame $Fr_\star$ such that $\eta(Fr_\star) > 0$.

The same argument shows that for any $v$-trajectory $\g \subset \mathsf{int}(X)$ and any $k$-frame $Fr_\star$ such that $\eta(Fr_\star)= 0$, the values $\{\tau\big(\phi^t_\ast(Fr_\star)\big)|_\g\}$ are $t$-independent. 
\hfill
\end{proof}

\begin{corollary}\label{cor.dtau(v) = eta} Let $X$ be a compact connected smooth manifold equipped with a non-vanishing vector field $v$. Let $\eta$ be a $v$-basic non-vanishing $k$-form, and $\tau$ a $v$-horizontal $k$-form on $X$ such that $v \, \rfloor \,d\tau = \eta$.  
\smallskip

Then $\tau$ is \emph{not}   bounded along any trapped $v$-trajectory $\g \subset \mathsf{int}(X)$ in the direction of any frame $Fr_\star$ such that $\eta(Fr_\star) > 0$. 
\smallskip

On the other hand, if a horizontal form $\tau$ is such that the form $v \, \rfloor \,d\tau$ is basic and $\tau$ is  bounded along any $v$-trajectory $\g \subset \mathsf{int}(X)$ in the direction of any frame $Fr_\star$, then $v$ does not have trapped trajectories, and $\d_1 X \neq \emptyset$. In such a case, $v$ admits a Lyapunov function.\end{corollary}

\begin{proof} Since $v \, \rfloor \,d\tau = \eta$ and $v \, \rfloor \,\tau =0$, we get $\mathcal L_{v}(\tau) = d(v \, \rfloor \,\tau) + v\, \rfloor \, d\tau = \eta.$ Thus Lemma \ref{lem.Lv(tau) = eta} applies.

In particular, if a horizontal form $\tau$ is such that the form $v \, \rfloor \,d\tau$ is basic and $\tau$ is  bounded along any $v$-trajectory $\g \subset \mathsf{int}(X)$ in the direction of any frame $Fr_\star$, then $v$ does not have trapped trajectories.

If a non-vanishing $v$-flow has no trapped trajectories, then its trajectories are closed segments (whose boundaries reside in $\d_1 X$) or singletons (residing in $\d_1 X$), which implies that $\d_1X \neq \emptyset$. By Lemma 5.6 from \cite{K3}, such a flow admits a Lyapunov function. 
\hfill
\end{proof}

As an instant application of Corollary \ref{cor.dtau(v) = eta}, we get the following claim.

\begin{corollary}\label{cor.tau_d(beta)_BOUNDED} Let $\b$ be a contact form, and $\tau$ a $(2k)$-form on a compact connected smooth manifold $X$, such that\footnote{This $\tau$ is horizontal, but not basic.} 
$$v_\b \, \rfloor \,\tau = 0 \text{\, and \,} v_\b \, \rfloor \,d\tau = (d \b)^k.$$ 

Then $\tau$ is not  bounded along each trapped $v_\b$-trajectory $\g \subset \mathsf{int}(X)$ in the direction of any frame $Fr_\star$ such that $(d \b)^k(Fr_\star) \neq 0$.
\smallskip


On the other hand, if $\tau$ is  bounded along each $v_\b$-trajectory $\g \subset \mathsf{int}(X)$ in the direction of any frame $Fr_\star$, then $v$ does not have trapped trajectories, and $\d_1 X \neq \emptyset$. In such a case, $v_\b$ admits a Lyapunov function. \hfill $\diamondsuit$
%

\end{corollary}

\begin{proposition} 
Let $\b$ be a contact form on a compact connected smooth manifold $X$. 
 Assume  that $H^{2k-1}_{\mathsf{basic}\,d\mathcal{R}}(X, v) = 0$ 
 and $[(d\b)^k] =0$  in $H^{2k}_{\mathsf{basic}\,d\mathcal{R}}(X, v)$. \smallskip

Then there exists a  $(2k-2)$-form $\tau$ on $X$ such that $v_\b \, \rfloor \,\tau = 0$ and $v_\b \, \rfloor \,d\tau = (d \b)^{k-1}$. \smallskip

If  this form $\tau$ is  bounded along each $v_\b$-trajectory $\g \subset \mathsf{int}(X)$ in the direction of any $(2k-2)$-frame $Fr_\star$, then $v_\b$ does not have trapped trajectories, and $\d_1 X \neq \emptyset$. In such a case, the Reeb vector field $v_\b$ admits a Lyapunov function. 
\end{proposition}

\begin{proof}
Since $d((d\b)^k)= 0$ and $v_\b\, \rfloor \, (d\b)^k =0$, we get that $(d\b)^k \in \Om^{2k}_{\mathsf{basic}\,d\mathcal{R}}(X, v_\b)$ is a basic cocycle. 

If $[(d\b)^k] =0$  in $H^{2k}_{\mathsf{basic}\,d\mathcal{R}}(X, v)$, then for $k > 0$, there is a basic $(2k-1)$-form $\a$ such that $d\a = (d\b)^k= d(\b \wedge (d\b)^{k-1})$. Thus, $d(\a - \b \wedge (d\b)^{k-1}) = 0$. 

Using that $H^{2k-1}_{\mathsf{basic}\,d\mathcal{R}}(X, v) = 0$, 
we conclude that there is a  horizontal $(2k-2)$-form $\tau$ such that $d\tau = \a - \b \wedge (d\b)^{k-1}$.
 Since $\a$ is a $v_\b$-basic form, 
  $$ v_\b \rfloor d\tau 
  = v_\b \rfloor \a   - v_\b \rfloor (\b \wedge (d\b)^{k-1})
  = - v_\b \rfloor (\b \wedge (d\b)^{k-1}) = (d\b)^{k-1}.$$ 

By the Cartan formula, 
$\mathcal L_{v_\b} \tau = d(v_\b \rfloor \tau)+ v_\b \rfloor d\tau =  (d\b)^{k-1}.$
Thus, Corollary \ref{cor.dtau(v) = eta} is applicable (replacing $k$ with $k-1$). We conclude that  if $\tau$ is  bounded along each Reeb trajectory $\g \subset \mathsf{int}(X)$ in the direction of any frame $Fr_\star$, then 
$v_\b$ does not have trapped trajectories and $\d_1 X \neq \emptyset$. In such a case, by Corollary \ref{cor.dtau(v) = eta}, $v_\b$ admits a Lyapunov function. 
\hfill
\end{proof}

  %
 






\section{Contact de Rham cohomology and Hodge structures transversal to Reeb foliations}

Switching gears, let us briefly describe some main features of the odd-dimensional Hodge theory, based on  structures  transversal to a given $1$-dimensional foliation. We follow closely the presentations in the papers of Z. He \cite{He}, Y. Lin \cite{Lin} and  of P. Ra\'{z}ny \cite{Raz}.
\smallskip

We denote by $\xi_{\b}$ the distribution $\ker(\b)$, transversal to the Reeb $1$-foliation $\mathcal F(v_\b)$. The $2$-form $d\b$ defines a symplectic structure on the bundle $\xi_{\b}$. As before, all these structures are defined in the vicinity of a compact manifold $X$  contained in an equidimentional  open manifold $\hat X$. 

The non-degenerated $2$-form $d\b\big |: \xi_\b \times \xi_\b \to \R$ produces a bundle isomorphism $Q_{d\b}: \xi_\b \to \xi^\ast_\b$ which leads to an isomorphism $\L^p Q_{d\b}: \L^p\xi_\b \to \L^p(\xi^\ast_\b)$ for each $p$. Thus, we get a non-degenerated pairing $K^{p \ast}_{d\b}:  \L^p\xi^\ast_\b \times \L^p\xi^\ast_\b \to \R$, produced by $(\L^p Q_{d\b})^{-1}$. Any basic $p$-form $\a$ can be viewed as a section of the bundle $\L^p\xi^\ast_\b$, since $v_\b$ belongs to the kernel of $\a$. 
Therefore, we get a non-degenerated pairing $K^{p \ast}_{d\b}: \Om^p_{\mathsf b}(X, v_\b) \times \Om^p_{\mathsf b}(X, v_\b) \to \R$ of basic $p$-forms. 
\smallskip

The {\sf basic Hodge star operator} $$\star_{\mathsf b}: \Om^p_{\mathsf b}(X, v_\b) \to \Om^{2n -p}_{\mathsf b}(X, v_\b)$$ 
is defined by the formula
\begin{eqnarray}\label{eq.BASIC_HODGE_STAR}
\kappa \wedge (\star_{\mathsf b} \rho) = K^{p \ast}_{d\b}(\kappa, \rho) \cdot \frac{(d\b)^n}{n!}.
\end{eqnarray}
for any pair of basic forms $\kappa, \rho \in \Om^p_{\mathsf b}(X, v_\b)$.\smallskip

The {\sf basic co-derivative operators} are introduced by the formula
\begin{eqnarray}
\delta_{\mathsf b} &=_{\mathsf{def}}\,&  (-1)^{p+1} (\star_{\mathsf b})\,\circ\, d\,\circ (\star_{\mathsf b}),
\end{eqnarray}
and the {\sf basic Laplace operators} by  the formula
\begin{eqnarray}
\Delta_{\mathsf b} &=_{\mathsf{def}}\,& d \,\circ \delta_{\mathsf b}\, + \, \delta_{\mathsf b} \,\circ  d.\label{eq.basic_LAPLACE}
\end{eqnarray}

We call a basic form $\a$ {\sf basically harmonic} if $\Delta_{\mathsf b}(\a) =0$. If $\a$ is closed, i.e., $d\a = 0$, and {\sf co-closed},  i.e., $\delta_{\mathsf b} \a = 0$, then the form $\a$ is evidently basically harmonic. If $\d X \neq \emptyset$, the converse may be not true due to boundary effects.
\smallskip

Since $d$ and $\Delta_{\mathsf b}$ commute, the basic harmonic forms form a subcomplex $\big(\mathsf{Harm}^\ast_{\mathsf b}(X, v), d\big)$ of the complex 
$\big(\Om^\ast_{\mathsf{basic}\,d\mathcal{R}}(X, v), d\big)$. (For a closed $X$, the differential $d$ in this complex is trivial.) Therefore, we may introduce also $H^k\big(\mathsf{Harm}^\ast_{\mathsf b}(X, v), d\big)$, the {\sf basic harmonic $k$-homology} of $(X, v)$. 
\smallskip

Based on the main theorem of \cite{CTGM}, we propose the following conjecture, in which the notion of a traversing vector field $v$ is equivalent  to the existence of a Lyapunov function for $v$. See \cite{K1} or \cite{K3} for the notion of a boundary generic vector field.
\begin{conjecture}\label{conj.HARMONIC_with_boundary} For a traversing and boundary generic vector field $v$ on a connected compact smooth manifold $X$ with boundary, 
$$H^k\big(\mathsf{Harm}^\ast_{\mathsf b}(X, v), d\big) \approx H^k(X; \R) \oplus H^{k-1}(X; \R)$$  
$$\approx H^k(\mathcal T(v); \R) \oplus H^{k-1}(\mathcal T(v); \R),$$
where $\mathcal T(v)$ is the trajectory space of the $v$-flow.
\end{conjecture}

Next, we apply the general results from \cite{EKA}, \cite{EKH}, and \cite{He},  to the case of $1$-dimensional Reeb foliations $\mathcal F(v_\b)$ on \emph{closed} manifolds. Unfortunately, no similar  theories have been developed for (singular) foliations on compact manifolds \emph{with boundary}, although the contours of such theories (in the spirit of \cite{CTGM} and Conjecture \ref{conj.HARMONIC_with_boundary}) are quite clear... \smallskip

For contact structures on compact manifolds with \emph{no boundary}, this application leads directly to the following assertions.

\begin{theorem}\label{th.semi-Kahler structure} 
Let $\b$ be a contact form on a \emph{closed} orientable $(2n+1)$-manifold $X$. 
\begin{itemize}
\item There is an isomorphism $$H^\ast_{\mathsf{basic}\,d\mathcal{R}}(X,  v_\b) \approx \ker(\Delta^\ast_{\mathsf b}(v_\b)),$$
implying that the basic de Rham cohomology is a finite-dimensional space. 
\smallskip

\item The basic 
Hodge star operator induces a duality isomorphism 
 $$\star_{\mathsf b} : H^k_{\mathsf{basic}\,d\mathcal{R}}(X, v_\b) \approx H^{2n-k}_{\mathsf{basic}\,d\mathcal{R}}(X, v_\b)$$
by taking each homology class to the harmonic representative of its image. 
\end{itemize}
\end{theorem}

Combining the result of Y.G. Oh \cite{Oh} with Theorem \ref{th.semi-Kahler structure} and using that the Laplacian operator commutes with the Hodge star-operator, we get the following claim.

\begin{corollary} For a \emph{generic} contact form $\b$ (i.e., for a residual set of contact forms) on a closed orientable $(2n+1)$-dimensional smooth manifold $X$, 
$$\ker(\Delta^{2n+1}_{\mathsf b}(v_\b)) \approx H^{2n+1}_{\mathsf{basic}\,d\mathcal{R}}(X, v_\b) \approx\; \R\; \approx H^{2n}_{\mathsf{basic}\,d\mathcal{R}}(X, v_\b) \approx \ker(\Delta^{2n}_{\mathsf b}(v_\b)).$$
\hfill $\diamondsuit$
\end{corollary}

 Theorem \ref{th.semi-Kahler structure} in combination with Proposition \ref{prop.ISO_de Rham} (or with Theorem \ref{th.CHAIN_equivalent}) leads instantly to the following claim.

\begin{theorem}\label{th.ker_Delta_under_deformations}
Let $\b$ and $\b_1 = \b +df$ be two contact forms on a closed  
smooth manifold $X$, where $df(v_\b) > -1$. 
\smallskip

Then the spaces $\ker(\Delta^k_{\mathsf{b}}(\b_1))$  and $\ker(\Delta^k_{\mathsf{b}}(\b))$ of basically harmonic $k$-forms of the contact forms $\b_1$ and $\b$ are canonically isomorphic for all $k$. 
\hfill $\diamondsuit$
\end{theorem}


Applying Theorem \ref{th.semi-Kahler structure} to the basic closed form $(d\b)^k$, leads to  the following assertion.

\begin{corollary}\label{cor.basic_Hodge_decomposition} 
%
Let $\b$ be a contact form on a closed smooth manifold $X$.\smallskip

Then the $v_\b$-basic form $(d\b)^k$, $k \leq n$, decomposes canonically as the sum of two basic $2k$-forms: 
 \begin{eqnarray}\label{eq.basic_Hodge_decomposition} 
 (d\b)^k = h_{(d\b)^k} + d\theta,
 \end{eqnarray}
where the form $h_{(d\b)^k}$ is basic and harmonic (that is, $\Delta_{\mathsf b}(h_{(d\b)^k}) = 0$) 
and $\theta$ is a basic $(2k-1)$-form\footnote{Recall that $\b$ is not a basic form.} whose differential $d\theta$ depends only on $\b$.
\smallskip

If the ``harmonic" elements $$[h_{(d\b)^k}] \in H^{2k}_{\mathsf{basic}\,d\mathcal{R}}(X, v_\b) \text{\; or \;} [\star_{\mathsf b}(h_{(d\b)^k})] \in H^{2n-2k}_{\mathsf{basic}\,d\mathcal{R}}(X, v_\b)$$ are nontrivial, then so are the elements $[(d\b)^k]  \in H^{2k}_{\mathsf{basic}\,d\mathcal{R}}(X, v_\b)$ and $[\star_{\mathsf b} (d\b)^k] \in \hfill \break H^{2n-2k}_{\mathsf{basic}\,d\mathcal{R}}(X, v_\b)$.
\hfill $\diamondsuit$
\end{corollary}

Recall that, in the classical K\"{a}hler geometry on even-dimensional manifold $X$, the volume of a complex $k$-dimensional submanifold $Y \subset X$ is given by the Wirtinger's formula $vol(Y) = \frac{1}{k!} \int_Y \om^k$, where $\om$ is the symplectic $2$-form on $X$. In particular, $ \int_Y \om^k > 0$. In sharp contrast with such a phenomenon, in the contact geometry on an odd-dimensional closed $X$,  for Hodge structures transversal to the Reeb foliation $\mathcal F(v_\b)$ and $\om = d\b$, we get the following properties.

\begin{corollary} Let the manifold $X$ be as in Corollary \ref{cor.basic_Hodge_decomposition} and let $\Sigma \subset X$ be a compact oriented submanifold of dimension $2k$. 
\begin{itemize}
\item Assuming that $\Sigma$ is closed, 
the ``{\sf harmonic volume}" and the ``{\sf symplectic volume}"of $\Sigma$ both vanish: 
$$\int_\Sigma h_{(d\b)^k} = \int_\Sigma (d\b)^k = 0$$

\item If the boundary $\d\Sigma$ of $\Sigma$ is tangent to the contact distribution $\xi_\b$, then for $2k-1 \leq n$ we get:
$$\int_\Sigma h_{(d\b)^k} = - \int_{\d\Sigma} \theta.$$
\end{itemize}
\end{corollary}

\begin{proof} The first claim follows by applying Stokes' theorem to formula (\ref{eq.basic_Hodge_decomposition}): $$\int_\Sigma (d\b)^k = \int_\Sigma d(\b \wedge (d\b)^{k-1}) = \int_{\d\Sigma}\b \wedge (d\b)^{k-1} = \int_{\emptyset}\b \wedge (d\b)^{k-1} = 0
$$ and using also that $\int_\Sigma d\theta = 0$  for a closed $\Sigma$. 
\smallskip

If the boundary $\d\Sigma$ 
of an oriented compact manifold $\Sigma$  is tangent to the distribution $\xi_\b$ (i.e., if $\d\Sigma$ a Legandrian submanifold), then $\b|_{\d\Sigma}= 0$ and
$$0 =\int_{\d\Sigma}\b \wedge (d\b)^{k-1} = \int_\Sigma (d\b)^k = \int_\Sigma (h_{(d\b)^k} + d\theta) = \int_\Sigma h_{(d\b)^k} + \int_{\d\Sigma} \theta.$$ 
Therefore,
$$\int_\Sigma h_{(d\b)^k} = - \int_{\d\Sigma} \theta = -\int_\Sigma d\theta.$$
\hfill
\end{proof}

By the formula above, the integral $\int_\Sigma d\theta$ does not depend on the choice of the submanifold $\Sigma$, whose Legandrian boundary $\d\Sigma$ is fixed. Also, the integral $\int_{\d\Sigma} \theta$ does not change if one replaces $\theta$ with $\theta + d\eta$. Thus, we get the following claim.

\begin{corollary} Let $\b$ be a contact form on a closed orientable 
$(2n+1)$-manifold $X$. 
\smallskip

Then, for  $2k-1 \leq n$, we get a map $$\Big(\int \theta\Big)_k:\;\; \big\{\mathsf{bounded\, Legandrian \, (2k-1)\text{\sf{-submanifolds}} \; d\Sigma\; \text{\sf of } \, } \mathsf X \big\} \longrightarrow \R,$$ which depends on the contact form $\b$ only.  \hfill $\diamondsuit$
\end{corollary} 

Note that, if the $\big(\int \theta\big)_k$-image of a Legandrian submanifold $\d\Sigma \subset X$ is negative, then the harmonic volume of an oriented manifold $\Sigma$ is positive.
\smallskip


Next, we consider a $v_\b$-invariant Riemannian metric $g$ on $TX$ and a $v_\b$-invariant almost complex structure  $J_{\mathsf b}: \xi_{\b} \to \xi_{\b}$ on the contact ditribution, where $J_{\mathsf b}^2 = -\mathsf{id}$, that satisfy the following compatibility properties (they define a {\sf basic almost-K\"{a}hler structure}):
\begin{eqnarray}
&v_\b\, \perp_{ g}\, \xi_{\b}, \label{eq.compatibility_A}\\
&\|v_\b\|_{g} = 1, \label{eq.compatibility_B} \\
&d\b(\sim,\, \sim) = g(J_{\mathsf b} \sim,\, \sim) = d\b(J_{\mathsf b} \sim,\, J_{\mathsf b} \sim) \;\text{ on }\; \xi_{\b}. \label{eq.compatibility_C}
\end{eqnarray} 

For $\dim(X) = 2n+1$, any Riemannian metric $g$ on $X$ defines the Hodge operator $$\star_{g}: \Om^p(X) \to \Om^{2n+1 -p}(X).$$ 

Similarly, following \cite{Raz}, we introduce the {\sf symplectic basic Hodge star operator} and {\sf symplectic basic coboundary operator}: 
\begin{eqnarray}\label{eq.symplectic Hodge}
&\star_{d\b} =_{\mathsf{def}}\; - J_{\mathsf b} \circ (\ast_{\mathsf b})\\
& \delta_{d\b} =_{\mathsf{def}}\; (-1)^{k+1} (\star_{d\b})\circ \,  d\, \circ\, (\star_{d\b}).  
\end{eqnarray}

The {\sf symplectic basic Laplace operator} is introduced by the familiar formula
$$\Delta_{d\b} = d \circ  \delta_{d\b} +  \delta_{d\b} \circ d.$$

The exterior multiplication of basic forms generates a multiplicative structure in vector space $\Om^\ast_{\mathsf{basic}\,d\mathcal{R}}(X, v)$. \smallskip 
The {\sf basic Lefschetz-Weyl operators} are defines by:
\begin{eqnarray}\label{eq.symplectic Lefschetz}
& L =_{\mathsf{def}}\; d\b\; \wedge \sim, \quad\quad  \L =_{\mathsf{def}}\; (\star_{d\b})\circ  L \circ (\star_{d\b}). 
\end{eqnarray}

The symplectic basic Hodge operator $\star_{d\b}$ maps $v_\b$-basic forms to $v_\b$-basic forms.
\begin{definition}
A $v_\b$-basic form $\a$ is said to be {\sf symplectically harmonic} if $d\a=0$ and \hfill \break $\delta_{d\b}\,\a =0$. \hfill $\diamondsuit$
\end{definition}

%
%


 
 For the Reeb foliation $\mathcal F(v_\b)$, Ra\'{z}ny's Theorem 2.17  \cite{Raz} can be reduced to the following claim:

 \begin{theorem}\label{th.Basic_hard_Lefschetz} For a closed 
 oriented $(2n+1)$-dimensional manifold $X$, equipped with a contact $1$-form $\b$ and an operator $J_{\mathsf b}: \xi_\b \to \xi_\b$ as in (\ref{eq.compatibility_C}) which defines an almost complex structure on  the contact distribution $\xi_\b$, the following two claims are equivalent:
 \begin{itemize}
 \item {\sf (Basic Hard Lefschetz property)}
 The multiplication by $(d\b)^k$ map 
 $$L^k: H^{n-k}_{\mathsf{basic}\,d\mathcal{R}}(X, v_\b) \to H^{n+k}_{\mathsf{basic}\,d\mathcal{R}}(X, v_\b)$$ is an isomorphism for all $k \leq n$, \smallskip
 
 \item Every basic cohomology class has a $\delta_{d\b}$-closed representative. \hfill $\diamondsuit$
 \end{itemize} 
 \end{theorem}
 
Note that,  by Proposition \ref{prop.[dBETA]not_zero}, the homology class $[d\b] \in H^2_{\mathsf{basic}\,d\mathcal{R}}(X, v_\b)$ is zero if and only if $v_\b$ admits a closed $1$-form $\a$ such that $\a(v_\b) = 1$. Therefore, for such $(X, \b)$, the Basic Hard Lefschetz property fails!

 
 
 \begin{definition}\label{def.PRIMITIVE} Let $X$ be an open oriented $(2n+1)$-dimensional manifold, equipped with a contact $1$-form $\b$ and an operator $J_{\mathsf b}: \xi_\b \to \xi_\b$ as in (\ref{eq.compatibility_C}). 
 
 $\bullet$ A $v_\b$-basic $(n-k)$-form $\a$ 
 is called {\sf primitive} if $$L^{k+1}\a =_{\mathsf{def}}\; \a \wedge (d\b)^{k+1} = 0.$$ 
 Equivalently, a $v_\b$-basic form $\a$ is {\sf primitive} if $\L \a = 0$.
 \smallskip
 
 $\bullet$  Similarly, a $v_\b$-basic cohomology class $[\a]$ of degree $(n-k)$ is {\sf primitive} if $L^{k+1}[\a]  = 0$. 

\hfill $\diamondsuit$
 \end{definition}
 
  Note that, on a closed $X$, $d\b$ is not primitive since $(d\b)^n \neq 0$.
  \smallskip

 Applying Proposition 2.19  \cite{Raz} to our setting, we get the Lefschetz decomposition of basic forms:
 
\begin{proposition}\label{prop.PRIMITIVE_decomposition} Given a closed oriented smooth $(2n+1)$-dimensional manifold $X$, equipped with a contact $1$-form $\b$ and an operator $J_{\mathsf b}: \xi_\b \to \xi_\b$ as in (\ref{eq.compatibility_C}), then any basic $k$-form $\a$ can be uniquely decomposed as 
$$\a =  \sum_{i=0}^{\lceil k/2 \rceil} \rho_{k -2i}  \wedge (d\b)^i,$$
where $\{\rho_{k -2i}\}$ are $v_\b$-basic \emph{primitive} $(k -2i)$-forms, given by the formula 
$$ \rho_{k -2i} = \Big(\sum_r \frac{a_{i, r}}{r!}\, L^r \circ \L^{i+r}\Big) \a.$$
Here the constants $a_{i, r}$ depend only on $(n, i, r)$. \hfill $\diamondsuit$
 \end{proposition}
 
As Theorem 2.20 from  \cite{Raz}, Proposition \ref{prop.PRIMITIVE_decomposition} leads to the following claim:
 
 \begin{theorem}\label{th.PRIMITIVE_decomposition}  Assume that a closed oriented smooth $(2n+1)$-dimensional manifold $X$, equipped with a contact $1$-form $\b$ and an operator $J_{\mathsf b}: \xi_\b \to \xi_\b$ as in (\ref{eq.compatibility_C}), is such that the Reeb foliation $\mathcal F(v_\b)$ satisfies the Basic Hard Lefschetz property. 
 Let $[\a]$ be a basic cohomology class of degree $k$. \smallskip
 
 Then $[\a]$ can be uniquely decomposed as 
 $$[\a] =  \sum_{i=0}^{\lceil k/2 \rceil} [\rho_{k -2i}]  \wedge [d\b]^i,$$
 where $\{[\rho_{k -2i}]\}$ are $v_\b$-basic \emph{primitive} $(k -2i)$-cohomology classes.
 \hfill $\diamondsuit$
 \end{theorem}
 
In turn, Ra\'{z}ny's Theorem 2.20, \cite{Raz}, leads to another important conceptual result, Corollary 4.4 \cite{Raz}. It claims (crudely) that homeomorphic manifolds, both satisfying the Basic Hard Lefschetz property for the appropriate foliations of the same dimension, have isomorphic basic de Rham cohomologies! Applying this result to Reeb foliations, we get the following assertion:
 
 \begin{theorem}\label{th.Lefschetz_implies_iso_homology} 
 Let $\b_0, \b_1$ be two contact forms on a closed manifold $X$ and  $J_{\mathsf b_0},  J_{\mathsf b_1}$ be two almost complex structure operators as in (\ref{eq.compatibility_C}), so that both Reeb foliations, $\mathcal F(v_{\b_0})$ and $\mathcal F(v_{\b_1})$, posess the Basic Hard Lefschetz property. \smallskip
 
Then, for each $k$, their basic de Rham $k$-cohomologies  are isomorphic $$H^k_{\mathsf{basic}\,d\mathcal{R}}(X,  v_{\b_1}) \approx H^k_{\mathsf{basic}\,d\mathcal{R}}(X,  v_{\b_0}),
$$ 
 and so are the spaces of basic symplectically harmonic $k$-forms:
 $$\ker(\Delta_{d\b_1}^k(v_{\b_1})) \approx \ker(\Delta_{d\b_0}^k(v_{\b_0})). \quad \diamondsuit$$
 \end{theorem}
 
 

\begin{proposition}\label{prop.induced_by_Psi} Let $X, Y$ be equidimensional compact smooth manifolds, and $\Psi: X \to Y$ a smooth immersion. Let $v_Y \neq 0$ be a vector field on $Y$. Put $v_X =_{\mathsf{def}} \Psi^\dagger(v_Y)$, the transfer of $v_Y$ to $X$. 
\smallskip

Then $\Psi$ induces  homomorphisms $$\Psi^\ast_{\mathsf b}: \Om^\ast_{\mathsf{basic}\,d\mathcal{R}}(Y, v_Y) \to \Om^\ast_{\mathsf{basic}\,d\mathcal{R}}(X, v_X)$$ of the basic differential complexes. Their basic de Rham and regular de Rham cohomologies form a commutative diagram:
\begin{eqnarray}
\Psi^\ast_{{\mathsf b}, H}:& H^\ast_{\mathsf{basic}\,d\mathcal{R}}(Y, v_Y) \to H^\ast_{\mathsf{basic}\,d\mathcal{R}}(X, v_X). \nonumber\\
\quad\quad \quad\quad \quad \quad\quad \quad & \downarrow  \quad \quad \quad \quad \downarrow                                                                          \nonumber\\
\Psi^\ast_H:& H^\ast_{d\mathcal{R}}(Y)\quad \quad \to \quad H^\ast_{d\mathcal{R}}(X). \nonumber
\end{eqnarray} 

If $v_Y$  admits a Lyapunov function, then so does $v_X$, and the lower row of the diagram can be replaced by the $\Psi$-induced homomorphism in the singular homology
$$\Psi^\ast_{\mathcal T,\, H}: H^\ast(\mathcal T(v_Y); \R) \to H^\ast(\mathcal T(v_X); \R)$$
of  the trajectory spaces of $v_Y$ and $v_X$ flows.
\end{proposition}

\begin{proof} 
The first claim follows from 
the naturality of the boundary operator $d$ and the operator $v \rfloor $ under smooth transformations. 

The second claim is implied by \cite{K1}, Theorem 5.1,  which asserts that, for a  vector field $v$ that admits a Lyapunov function, the map $X \to \mathcal T(v)$ is a homology equivalence. 
\hfill
\end{proof}

For a contact form $\b$ and its Reeb vector field $v_\b$, the form $(d\b)^\ell \in \Om^{2\ell}(X, v_\b)$ is closed, $v_\b$-invariant, and horizontal. Therefore, it is basic and determines an element $[(d\b)^\ell] \in H^{2\ell}_{\mathsf{basic}\,d\mathcal{R}}(X, v_\b)$. 
%
Note that if $[(d\b)^\ell] = 0$, then  $[(d\b)^k] = 0$ for all $k > \ell$. 

\begin{lemma}\label{lem.naturality_in_DR} Let $X, Y$ be equidimensional compact smooth manifolds that carry contact forms $\b_X, \b_Y$. Let $\Psi: X \to Y$ be a smooth immersion such that $\Psi^\ast(\b_Y) = \b_X$. 

Then, for any $\ell \in [0, n]$, we have the equality $$\Psi^\ast([(d\b_Y)^\ell]) = [(d\b_X)^\ell]$$ in $H^{2\ell}_{\mathsf{basic}\,d\mathcal{R}}(X, v_X)$.
\end{lemma}

\begin{proof}
The property $\Psi^\ast(\b_Y) = \b_X$ implies that $\Psi_\ast$ maps the Reeb vector field $v_{\b_X}$ the to the Reeb vector field $v_{\b_Y}$. By naturality, the pull-back of a $v_{\b_Y}$-basic differential form on $Y$ is a $v_{\b_X}$-basic differential form on $X$, and the differential operators $d_X$ and $d_Y$ on forms commute with $\Psi^\ast$. Therefore, $\Psi^\ast([(d\b_Y)^\ell]) = [(d\b_X)^\ell]$ in $H^{2\ell}_{\mathsf{basic}\,d\mathcal{R}}(X, v_X)$.
\hfill
\end{proof}

\begin{corollary} Under the hypotheses of Lemma \ref{lem.naturality_in_DR}, if $[(d\b_X)^\ell] \neq 0$ in $H^{2\ell}_{\mathsf{basic}\,d\mathcal{R}}(X, v_{\b_X})$ and $[(d\b_Y)^k] = 0$ in $H^{2k}_{\mathsf{basic}\,d\mathcal{R}}(Y, v_{\b_Y})$ for $k \leq \ell$, then no immersion $\Psi$, subject to $\Psi^\ast(\b_Y) = \b_X$, exists. 
\hfill $\diamondsuit$
\end{corollary}

Let us conclude with a conjecture, weakly motivated by \cite{Raz}, Corollary 4.4. Recall that the corollary claims that \emph{homeomorphic} manifolds, both satisfying the Basic Hard Lefschetz property for the appropriate foliations of the same dimension, have isomorphic basic de Rham cohomologies. In other words, for closed smooth manifolds that satisfy the Basic Hard Lefschetz property, the basic de Rham cohomology is a topological invariant! 

\begin{conjecture} Let $X$ be a compact connected smooth manifold and $v$ a (non-vanishing) vector field on $X$. Assume that $v$ admits a Lyapunov function and is transversally generic in the sense of \cite{K1} (see also \cite{Mo}). 

Then the $v$-basic de Rham cohomology $H^\ast_{\mathsf{basic}\,d\mathcal{R}}(X, v)$ depends only on the stratified topological type of the trajectory space $\mathcal T(v)$. 
\hfill $\diamondsuit$
\end{conjecture}



\begin{thebibliography}{30}

\bibitem [CTGM]{CTGM} Cappell, S., DeTurk, D., Gluck, H., Miller, E., {\it Cohomology of Harmonic Forms on Riemannian Manifolds With Boundary}, arXiv:math/0508372v1 [math.DG], 19 Aug 2005.

\bibitem [EKH]{EKH} El Kacimi-Alaoui, A., Hector, G. {\it D\"{e}composition de Hodge basique pour un feuilletage riemannien}, Ann. Inst. Fourier 36, 207-227 (1987).

\bibitem [EKA]{EKA} El Kacimi-Alaoui, A., {\it Op\'{e}rateurs transversalement elliptiques sur un feuilletage riemannien et applications}. Compos. Math. 73, 57-106 (1990).

\bibitem [F]{F} Farber, M., {\it Topology of Closed One-Forms}, Mathematical Surveys and Monographs 108, AMS, 2004.


\bibitem [Good]{Good} Goodman, S., {\it Vector Fields with Transversal Foliations}, Topology, vol.24, no. 3, 333-340 (1985). 

\bibitem [Gray]{Gray} Gray, J.W., {\it Some global properties of contact structures}, Ann. of Math. (2) 69 (1959), 421-450.

\bibitem [He]{He} He, Z., {\it Odd dimenisonal symplectic manifolds}, MIT Ph.D thesis, 2010.

\bibitem[K1]{K1} Katz, G., {\it  Traversally Generic \& Versal Flows: Semi-algebraic Models of Tangency to the Boundary}, Asian J. of Math., vol. 21, No. 1 (2017), 127-168 (arXiv:1407.1345v1 [mathGT] 4 July, 2014)).



\bibitem [K2]{K2} Katz, G., {\it Causal Holography in Application to the Inverse Scattering Problem}, Inverse Problems and Imaging J., June 2019, 13(3), 597-633 (arXiv: 1703.08874v1 [Math.GT], 27 Mar 2017).


\bibitem [K3]{K3} Katz, G., {\it Morse Theory of Gradient Flows, Concavity, and Complexity on Manifolds with Boundary}, World Scientific, (2019). ISBN 978-981-4368-75-9

\bibitem [K4]{K4} Katz, G.,  {\it Causal Holography of Traversing Flows},  Journal of Dynamics and Differential Equations (2020)  https://doi.org/10.1007/s10884-020-09910-y 


\bibitem [K5]{K5} Katz, G., {\it Recovering Contact Forms from Boundary Data}, arXiv:2309.14604 v3 [math.SG] 11 Mar 2024.

\bibitem [K6]{K6} Katz, G., {\it Contact de Rham Cohomology and Hodge Structures Transversal to the Reeb Foliations}, arXiv: 2502.0977.09773v1 [math.DG] 13 Feb 2025.






\bibitem [Lin]{Lin} Lin, Y., {\it Lefschetz Contact Manifolds and Odd Dimensional Symplectic Geometry}, arXiv: 1311.1431v4 [math SG] 2 Sep 2016.


\bibitem [Mo]{Mo} Morse, M. {\it Singular points of vector fields under  general boundary conditions}, Amer. J. Math. 51 (1929), 165-178.

\bibitem [Oh]{Oh} Oh, Y.G., {\it Foliation de Rham cohomology of Generic Reeb Foliations}, arXiv:2504.16453v2 [math.SG] 12 May 2025.



\bibitem [Raz]{Raz} Ra\'{z}ny, P., {\it Cohomology of Manifolds with Structure Group $\mathsf{U}(n) \times \mathsf{O}(s)$}, Geometriae Dedicata (2023) 217:58, https://doi.org/10.1007/s10711-023-00796-w.

\bibitem [Sch]{Sch} Schwartzman, {\it Asymptotic Cycles}, Ann. of Math., 66 (1957),  270-284.




\end{thebibliography}
\end{document}